\documentclass[11pt]{amsart}
\usepackage[new]{old-arrows}
\usepackage[T1]{fontenc}
\usepackage[english]{babel}
\usepackage{amsmath,amssymb,amsthm,enumerate,xspace}
\usepackage{comment}
\usepackage{accents}
\usepackage[colorlinks=true,citecolor=blue,linkcolor=blue]{hyperref}
\usepackage[all]{xy}
\usepackage{tikz}
\usepackage{mathrsfs}
\usepackage{cleveref}

\numberwithin{equation}{section}

\newtheorem*{thm*}{Theorem}

\newtheorem*{prob*}{Problem}

\newtheorem*{lem*}{Lemma}

\newtheorem*{iprob*}{Problem}

\theoremstyle{definition}

\newtheorem*{defi*}{Definition}

\newtheorem*{acks*}{Acknowledgements}

\newcommand{\ZZ}{\mathbf{Z}}
\newcommand{\PP}{\mathbf{P}}
\newcommand{\RR}{\mathbf{R}}
\newcommand{\NN}{\mathbf{N}}

\newcommand{\CC}{\mathbf{C}}
\newcommand{\QQ}{\mathbf{Q}}

\newcommand{\SL}{\mathbf{SL}}

\newcommand{\se}{\subseteq}

\newcommand{\inv}{^{-1}}

\newcommand{\lra}{\longrightarrow}

\newcommand{\one}{\boldsymbol{1}}

\DeclareMathOperator{\prim}{Prim}

\newcommand{\tto}{\twoheadrightarrow}
\newcommand{\cstar}{\mathrm{C}^*}
\newcommand{\cmax}{\cstar_{\mathrm{max}}}
\newcommand{\cred}{\cstar_{\mathrm{red}}}

\title[A family of exotic group C*-algebras]{A family of exotic group C*-algebras}

\author[Maria Gerasimova]{Maria Gerasimova}
\author[Nicolas Monod]{Nicolas Monod}

\address{M.G.: University of Münster, Mathematical Institute, Einsteinstraße 62, 48149 Münster, Germany}
\email{mari9gerasimova@mail.ru}
\address{N.M.: \'Ecole Polytechnique Fédérale de Lausanne (EPFL), CH–1015 Lausanne, Switzerland}
\email{nicolas.monod@epfl.ch}
\begin{document}

\begin{abstract}
We show that a large family of groups without non-abelian free subgroups satisfy the following strengthening of non-amenability: they each have a rich supply of irreducible representations defining exotic C*-algebras. The construction is explicit.
\end{abstract}
\maketitle



\section{Introduction}
\subsection{Group C*-algebras}
Let $\Gamma$ be any group; in this note, we consider groups without topology. Two C*-algebras are canonically attached to $\Gamma$: the \textbf{maximal C*-algebra} $\cmax(\Gamma)$ and the \textbf{reduced C*-algebra} $\cred(\Gamma)$. Moreover, there is a canonical epimorphism $\cmax(\Gamma)\to \cred(\Gamma)$, which is an isomorphism if and only if $\Gamma$ is amenable.

Following~\cite{Kyed-Soltan}, \cite{Kaliszewski-Landstad-Quigg13}, \cite{Kaliszewski-Landstad-Quigg16}, \cite{Wiersma2016} and~\cite{Ruan-Wiersma}, a C*-algebra $A$ is called \textbf{exotic} if it lies strictly inbetween $\cmax(\Gamma)$ and $\cred(\Gamma)$. That is, if $\cmax(\Gamma)\to \cred(\Gamma)$ factors through two epimorphisms
\[
\cmax(\Gamma) \lra A \lra \cred(\Gamma)
\]
and neither of them is an isomorphism. Thus, the existence of an exotic C*-algebra is a refinement of the non-amenability of $\Gamma$.

The purpose of this note is to describe an explicit situation where $\Gamma$ admits an uncountable family of different exotic C*-algebras. Moreover, these algebras are defined by concrete \emph{irreducible} representations of $\Gamma$. Specifically, we consider the quasi-regular representation algebras associated to uncountably many suitable subgroups of $\Gamma$.

In particular, we obtain an explicit and simple parametrization of a huge set in the \textbf{primitive dual} $\prim(\Gamma)$ of the group $\Gamma$, reflecting layers upon layers of non-amenability in this dual. Moreover, in our situation $\Gamma$ is known to be \textbf{C*-simple}, which by definition means that $\cred(\Gamma)$ is simple, i.e. that the interval between $\cmax(\Gamma)$ and $\cred(\Gamma)$ is a maximal interval.

Of particular interest is the fact that our groups $\Gamma$ do not contain non-abelian free subgroups. First, since exotic algebras consitute a refinment of non-amenability, the von Neumann--Day problem naturally challenges us to find such examples. Secondly, some early constructions of exotic algebras (\cite{Brown-Guentner2013}, \cite{Okayasu14}) were precisely based on the analytical properties of non-abelian free groups and subgroups, namely on so-called $L^p$-representations (cf.\ also \cite{deLaat-Siebenand}).

\medskip
Very different examples can already be found in~\cite{Bekka-Kaniuth-Lau-Schlichting} and~\cite{Bekka99}; as they are based on non-solvable Lie groups and respectively their lattices, they happen to contain non-abelian free subgroups too. A rather different approach with remarkable properties of quasi-regular C*-algebras can be found in~\cite{Bekka-Kalantar}. Finally, a completely general observation from~\cite[Rem.~2.2]{Echterhoff-Quigg} is that whenever $\Gamma$ is a non-amenable non-Kazhdan group, the (very much non-irreducible) representation $\lambda_\Gamma\oplus\one$ generates an exotic C*-algebra $A$. Indeed, $A$ maps onto $\cred(\Gamma)$ by construction. This projection is non-isomorphic since $\one$ is not weakly contained in $\lambda_\Gamma$ by non-amenability (the Hulanicki--Reiter criterion). The fact that $A\cong \cmax(\Gamma)$ would imply Kazhdan's property is the so-called ``Kazhdan projection'' criterion, see~\cite[Lem.~3.1]{Valette84}.

\subsection{A family of groups}
Let $S\se \NN$ be any set of prime numbers and denote by $\ZZ[1/S]$ the ring of $S$-integers. Following~\cite{Monod_PNAS,Monod_piecewise_pre}, we consider the group $\Gamma_S$ of all piecewise-$\SL_2(\ZZ[1/S])$ homeomorphisms of the real line $\RR$. More precisely, $\Gamma_S$ consists of all homeomorphisms $g$ for which $\RR$ can be cut into finitely many intervals such that, on each interval, $g$ coincides with a projective transformation $x\mapsto \dfrac{a x+b}{c x+ d}$ for some $\begin{pmatrix}a & b\\ c & d\end{pmatrix}$ in $\SL_2(\ZZ[1/S])$.

We recall that $\Gamma_S$ is a ``free group free group''; in the notation of~\cite{Monod_PNAS}, $\Gamma_S=H(\ZZ[1/S])$, except for $S=\varnothing$ where the breakpoint conditions chosen in~\cite{Monod_PNAS} define a smaller group, see~\cite[\S2]{Monod_piecewise_pre}.

\medskip

If we fix $S$, any subset $T\se S$ provides a subgroup $\Gamma_T< \Gamma_S$. Starting with $S$ infinite, this is an uncountable family of subgroups of $\Gamma_S$. This poset, the collection of all subsets $T\se S$, will parametrise the region of interest in the primitive dual $\prim(\Gamma_S)$ under the map $T \mapsto \cstar(\lambda_{\Gamma_S/\Gamma_T})$ which takes $T$ to the C*-algebra generated by the quasi-regular representation $\lambda_{\Gamma_S/\Gamma_T}$ of $\Gamma_S$ associated to $\Gamma_T$. In other words, $\lambda_{\Gamma_S/\Gamma_T}$ is the $\Gamma_S$-representation induced from the trivial $\Gamma_T$-representation.

\begin{thm*}\label{thm:main}
For any non-empty sets $T\subsetneqq S$ of prime numbers,  we have an exotic algebra
\[
\cmax(\Gamma_S) \xrightarrow{\ \not\simeq\ } \cstar(\lambda_{\Gamma_S/\Gamma_T}) \xrightarrow{\ \not\simeq\ } \cred(\Gamma_S)
\]
with $\lambda_{\Gamma_S/\Gamma_T}$ irreducible and $\cred(\Gamma_S)$ simple.

Moreover, given $S$, the corresponding quotients $\cmax(\Gamma_S) \tto \cstar(\lambda_{\Gamma_S/\Gamma_T})$ are pairwise non-isomorphic as $T$ varies.
\end{thm*}

Since there is a correspondance between (non-degenerate) representations of $\cmax(\Gamma_S)$ and unitary $\Gamma_S$-representations~\cite[\S13]{Dixmier6996_C}, all the above statements on group C*-algebras can be reformulated in terms of weak containment and weak inequivalence of various $\Gamma_S$-representations. For instance, the last statement means that the various $\lambda_{\Gamma_S/\Gamma_T}$ are pairwise not weakly equivalent. 

In the opposite direction, the only part of the Theorem that we stated in terms of $\Gamma_S$-representations is the irreducibility, which alows us to view $\lambda_{\Gamma_S/\Gamma_T}$ as points in the unitary dual $\widehat \Gamma_S$. In the C* language, this amounts to saying that the ideal of $\cmax(\Gamma_S)$ defining $\cstar(\lambda_{\Gamma_S/\Gamma_T})$ is a \emph{primitive} ideal.

In conclusion, we have faithfully embedded the entire collection subsets $T\se S$ into the primitive dual $\prim(\Gamma_S)$ of $\Gamma_S$ and a fortiori in the unitary dual $\widehat \Gamma_S$ since $\prim(\Gamma_S)$ can be viewed as the Kolmogorov $T_0$-quotient of $\widehat \Gamma_S$. Moreover, this region of $\prim(\Gamma_S)$ consists entirely of exotic group C*-algebra of $\Gamma_S$.

\section{Proof of the Theorem}
\label{sec:proof}
We begin by recording a general property of piecewise-projective groups; similar facts were already observed in~\cite{Monod_PNAS}, \cite{Burillo-Lodha-Reeves} and~\cite{Brum-MatteBon-Rivas-Triestino_arx}. Given any ring $A<\RR$, we denote by $H_c$ the subgroup of \emph{compactly supported} piecewise-$\SL_2(A)$ homeomorphisms of $\RR$ and by $H'_c$ the derived subgroup of $H_c$.

\begin{lem*}
For any $h_0\in \SL_2(A)$ and any compact interval $I\se \RR$ with $\infty\notin h_0 I$, there is $h\in H'_c$ such that $h$ and $h_0$ coincide on $I$.
\end{lem*}

\begin{proof}
We first claim that there is $h_1\in H_c$ coinciding with $h_0$ on $I$. Write $I=[u,v]$; we shall construct $h_1$ on $[v, +\infty)$ with $h_1 v = h_0 v$, and then the same argument can be applied on $(-\infty, u]$ to complete the claim. If $h_0$ fixes $v$, we can continue with the identity. If $h_0 v > v$, we can pick $x\in (v, h_0 v)$ close enough to $v$ that $h_0 x \in (h_0 v, +\infty)$. We can choose a hyperbolic element $q \in \SL_2(A)$ whose repelling/attracting fixed points $\xi_-, \xi_+$ are respectively in $(v, x)$ and in $(h_0 x, +\infty)$. Indeed already for $\SL_2(\ZZ)$ the pairs of fixed points are dense in $\PP^1 \times \PP^1$. Upon replacing $q$ by a positive power of itself, $q x > h_0 x$. Since on the other hand $q \xi_- = \xi_- < h_0 v< h_0 \xi_-$, there is by continuity some $t\in(\xi_-, x)$ with $q t = h_0 t$. We can now define $h_1$ by $h_0$ on $[v,t]$, by $q$ on $[t, \xi_+]$ and by the identity on $[\xi_+, +\infty)$.

The case $h_0 v < v$ is analogous: pick $x>v$ close enough that $h_0 x \in (h_0 v, v)$ and choose $q$ such that $\xi_-\in(v, x)$ and $\xi_+\in (-\infty, h_0 v)$. Replacing $q$ by a suitable power, we have $q v < h_0 v$ but on the other hand  $q \xi_- = \xi_- > v > h_0 x > h_0 \xi_-$. Thus there is $t\in(v, \xi_-)$ with $q t = h_0 t$ and we define $h_1$ by $h_0$ on $[v,t]$, by $q$ on $[t, \xi_-]$ and by the identity on $[\xi_-, +\infty)$. The claim is established.

Let now $J$ be a compact interval containing $I$ and the support of $h_1$. We can choose an element $b_1\in H_c$ with $b_1\inv  J\cap J =\varnothing$; this exists e.g. by another application of the first claim, this time for $J$ and a translation $b_0\in \SL_2(A)$ that translates the left endpoint of $J$ past its right endpoint. Then $b_1 h_1\inv b_1\inv$ is trivial on $I$ and hence the commutator $h=h_1 b_1 h_1\inv b_1\inv$ in $H'_c$ has the desired properties.
\end{proof}

\subsection{Irreducibility}
\label{sec:irred}
It is well-known that the representation $\lambda_{\Gamma_S/\Gamma_T}$ is irreducible if and only if $\Gamma_T$ is self-commensurating in $\Gamma_S$. This is generally attributed to Mackey as it follows from Theorem~6' in~\cite{Mackey51}; we note that it was already proved by Godement in Appendice~A p.~80 of~\cite{Godement48}.

Thus, given any element $g\in \Gamma_S$ not in $\Gamma_T$, we need to show that $\Gamma_T^g \cap \Gamma_T$ does not have finite index in both $\Gamma_T$ and the conjugate $\Gamma_T^g$.

To this end, it suffices to find a subgroup $\Lambda<\Gamma_T$ without proper finite index subgroups, e.g. infinite and \emph{simple}, such that $\Lambda^g$ is not in $\Gamma_T$. We now proceed to show that the second derived subgroup $\Lambda=(\Gamma_T)''$ has the required properties.

The simplicity of $H(A)''$ and the identity $H(A)''=H(A)_c'$ hold for any ring $A<\RR$, see~\cite{Burillo-Lodha-Reeves}. In fact, all this holds more generally for all ``locally moving'' groups of homeomorphisms of $\RR$, see~\cite[\S4]{Brum-MatteBon-Rivas-Triestino_arx}. In our case, $\Gamma_T=H(A)$ with $A=\ZZ[1/T]$.

Since $g\notin \Gamma_T$, there is an interval $J$ on which $g$ is represented by a matrix $\begin{pmatrix}a & b\\ c&d\end{pmatrix}$ in $\SL_2(\ZZ[1/S])$ with at least one coefficient not in $A$. That is, some coefficient contains a negative power of some $p\in S\smallsetminus T$.

Suppose first that this coefficient is $a$ or $b$. Choose any $q\in T$ and select $n\in\NN$ large enough so that the matrix $h_0=\begin{pmatrix}1 & 0\\ q^{-n}&1\end{pmatrix}$ satisfies $h_0 I \se \mathrm{Int}(g J)$ for some compact interval $I\se g J$. This is possible since $h_0$ converges to the identity as $n\to+\infty$. Applying the Lemma, we obtain $h \in \Lambda$ given by $h_0$ on $I$. The conjugate $g\inv h g$ is given on $g\inv I\se J$ by a matrix of the form
\[
\begin{pmatrix}a & b\\ c&d\end{pmatrix}\inv \begin{pmatrix}1 & 0\\ q^{-n}&1\end{pmatrix} \begin{pmatrix}a & b\\ c&d\end{pmatrix} = \begin{pmatrix}* & -b^2 q^{-n}\\ a^2 q^{-n}&*\end{pmatrix}.
\]
Thus the negative power of $p$ is still present in that case and hence $g\inv h g$ is not in $\Gamma_T$.

If the coefficient is $c$ or $d$, then we argue similarly with $h_0=\begin{pmatrix}1 & q^{-n}\\ 0&1\end{pmatrix}$ and this time the conjugate $g\inv h_ g$ involves a matrix $\begin{pmatrix}* & d^2 q^{-n}\\ -c^2 q^{-n}&*\end{pmatrix}$, and hence again is not in $\Gamma_T$. This completes the proof of irreducibility.

\subsection{Unconfinment}
\label{sec:unconfinment}
Let $\Lambda<\Gamma$ be a subgroup of a group $\Gamma$. Following~\cite{Hartley-Zalesskii93}, recall that $\Lambda$ is called \textbf{unconfined} in $\Gamma$ if the closure in the Chabauty space of subgroups of $\Gamma$ of the conjugation $\Gamma$-orbit of $\Lambda$ contains the trivial subgroup. Explicitly, this simply means that for any finite subset $E\se \Gamma$ not containing the identity, there is $\gamma\in\Gamma$ such that the conjugate $\gamma \Lambda \gamma\inv$ does not meet $E$.

The relevance of this notion to our situation is that it implies that the quasi-regular representation $\lambda_{\Gamma/\Lambda}$ weakly contains the regular representation $\lambda_{\Gamma}$. Indeed, this follows from Fell's continuity of the induction map, see Theorem~4.2 in~\cite{Fell1964}.

\medskip

Given $T\subsetneqq S$, we shall now prove that $\Gamma_T$ is unconfined in $\Gamma_S$. Equivalently, we produce a sequence $(g_n)$ in $\Gamma_S$ such that for every non-trivial $h\in \Gamma_S$, the conjugate $g_n\inv h g_n$ is outside $\Gamma_T$ for all $n$ large enough (depending on $h$).

To this end, let $p\in S \smallsetminus T$ and define $g_n$ by the element $\begin{pmatrix}p^n & p^{-2 n}\\ 0&p^{-n}\end{pmatrix}$ of $\SL_2(\ZZ[1/S])$. Note that $g_n$ fixes $\infty$ and hence defines an element of $\Gamma_S$. Consider now any non-trivial $h\in \Gamma_S$. There is some interval $I\se \RR$ on which $h$ is represented by an element $h_I=\begin{pmatrix}a & b\\ c&d\end{pmatrix}$ of $\SL_2(\ZZ[1/S])$ which is not $\pm\mathrm{Id}$. On $g_n\inv I$, the conjugate $g_n\inv h g_n$ is represented by the conjugate of $h_I$, whose top right corner is computed to be
\[
b p^{-2n}  + (a-d) p^{-3n} - c p^{-4n}.
\]
If this number were in $\ZZ[1/T]$ for arbitrarily large $n$, then the three coefficients $b$, $a-d$ and $c$ would vanish because the exponents of $p^{-n}$ are different. Together with the determinant condition $ad-bc=1$, this implies $a^2=1$ and hence $h_I=\pm\mathrm{Id}$, contrary to our assumption.

\subsection{Inequivalence}
\label{sec:ineq}
We now justify that whenever $T, T'\se S$ are two different subsets, the quotients $\cmax(\Gamma_S) \tto \cstar(\lambda_{\Gamma_S/\Gamma_T})$ and $\cmax(\Gamma_S) \tto \cstar(\lambda_{\Gamma_S/\Gamma_{T'}})$ are non-isomorphic. Without loss of generality, we can assume $T\not\se T'$ and we shall verify the following more precise statement: there is no vertical morphism for which the following diagram commutes.
\[
\xymatrix{& &{\cstar(\lambda_{\Gamma_S/\Gamma_{T'}})}\ar@{-->}[dd]\\
{\cmax(\Gamma_S)} \ar@/^1pc/[rru] \ar@/_1pc/[rrd]& & \\
& &{\cstar(\lambda_{\Gamma_S/\Gamma_T})}}
\]
By the correspondance between representations of $\Gamma_S$ and of $\cmax(\Gamma_S)$, this is equivalent to the statement that $\lambda_{\Gamma_S/\Gamma_T}$ is not weakly contained in $\lambda_{\Gamma_S/\Gamma_{T'}}$.

Suppose for a contradiction that this weak containment holds. In particular, the restriction $(\lambda_{\Gamma_S/\Gamma_T})|_{\Gamma_T}$ to $\Gamma_T$ is weakly contained in the restriction $(\lambda_{\Gamma_S/\Gamma_{T'}})|_{\Gamma_T}$. But $(\lambda_{\Gamma_S/\Gamma_T})|_{\Gamma_T}$ contains the trivial $\Gamma_T$-representation. Thus this trivial representation is weakly contained in $(\lambda_{\Gamma_S/\Gamma_{T'}})|_{\Gamma_T}$. This is equivalent to stating that $\Gamma_{T'}$ is \emph{co-amenable to $\Gamma_{T}$ relative to $\Gamma_S$}, see~\cite{Monod_piecewise_pre}. However, it is proved  in~\cite{Monod_piecewise_pre} that this relative co-amenability does not hold, in fact not even relatively to the larger group $H(\QQ)$.

\subsection{End of proof}
The representation $\lambda_{\Gamma_S/\Gamma_T}$ defines a quotient $\cstar(\lambda_{\Gamma_S/\Gamma_T})$ of $\cmax(\Gamma_S)$. The fact that the canonical map $\cmax(\Gamma_S)\to \cred(\Gamma_S)$ factors through
\[
\cmax(\Gamma_S) \lra \cstar(\lambda_{\Gamma_S/\Gamma_T})\lra \cred(\Gamma_S) = \cstar(\lambda_{\Gamma_S})
\]
is equivalent to $\lambda_{\Gamma_S}$ being weakly contained in $\lambda_{\Gamma_S/\Gamma_T}$, which is established in \Cref{sec:unconfinment}.

For $\cstar(\lambda_{\Gamma_S/\Gamma_T})$ to be exotic, we need to know that neither of the above two morphisms is an isomorphism.

If the first morphism is an isomorphism, then $\cstar(\lambda_{\Gamma_S/\Gamma_T})$ has an epimorphism to the scalar algebra $\CC$ because $\cmax(\Gamma_S )$ admits such a morphism. This means that the trivial representation is weakly contained in $\lambda_{\Gamma_S/\Gamma_T}$, which happens if and only if $\Gamma_T$ is co-amenable in $\Gamma_S$, see~\cite[No.~3, \S2]{Eymard72}. However, it is shown in~\cite{Monod_piecewise_pre} that $\Gamma_T$ is not co-amenable in $\Gamma_S$.

The second morphism is an isomorphism if and only if $\lambda_{\Gamma_S/\Gamma_T}$ is weakly contained in $\lambda_{\Gamma_S}$. The latter holds if and only if $\Gamma_T$ is an amenable group, see Proposition~4.2.1 in~\cite{Anantharaman03}. (That reference requires another condition which is trivially satisfied in the current setting of discrete groups.) Thus, using the non-amenability established in~\cite{Monod_PNAS}, we conclude that indeed $\cstar(\lambda_{\Gamma_S/\Gamma_T})$ is exotic.

The fact that $\lambda_{\Gamma_S/\Gamma_T}$ is irreducible was proved above in \Cref{sec:irred} and the simplicity of $\cred(\Gamma_S)$ was established in~\cite{LeBoudec-MatteBon_C*s}. Finally, the inequivalence was proved in \Cref{sec:ineq}.\qed

\section{Comments}
If we only want a group $\Gamma$ without non-abelian free subgroups but admitting \emph{some} exotic group C*-algebra, then other easy examples from quasi-regular representations associated to subgroups $\Lambda<\Gamma$ can be constructed as follows. Of course, these simple examples will not enjoy the stronger properties listed in the Theorem, in particular the representations will be far from irreducible and therefore they do not describe anything in the dual or primitive dual of $\Gamma$.

For the reasons exposed in \Cref{sec:proof}, we will have an exotic algebra
\[
\cmax(\Gamma) \xrightarrow{\ \not\simeq\ } \cstar(\lambda_{\Gamma/\Lambda}) \xrightarrow{\ \not\simeq\ } \cred(\Gamma)
\]
provided the following three conditions are all satisfied:

\begin{enumerate}[(i)]
\item $\Lambda$ is not amenable;\label{pt:amen}
\item $\Lambda$ is unconfined in $\Gamma$;\label{pt:unco}
\item $\Lambda$ is not co-amenable in $\Gamma$.\label{pt:co-amen}
\end{enumerate}

\medskip
\noindent
Start with any non-amenable group $\Lambda$ without non-abelian free subgroups. Consider the ``lamplighter'' restricted wreath product
\[
\Gamma = \Lambda \wr Z = \Big( \bigoplus_{z\in Z} \Lambda_z \Big) \rtimes Z
\]
where $Z$ is any infinite group without non-abelian free subgroups; e.g.\ $Z=\ZZ$ or $Z=\Lambda$. Here $\Lambda_z$ denotes a copy of $\Lambda$ for each $z\in Z$. Note that $\Gamma$ still has no non-abelian free subgroups. View $\Lambda$ as a subgroup of $\Gamma$, say $\Lambda= \Lambda_e$ at the coordinate $e\in Z$.

\medskip
Condition~\eqref{pt:amen} holds by construction. For~\eqref{pt:unco}, let $(z_n)_{n\geq 1}$ be any sequence in $Z$ leaving any finite subset, viewed as a sequence in $\Gamma$. Then $\Lambda^{z_n}$ converges to the trivial subgroup in the Chabauty space because $\Lambda^{z_n}=\Lambda_{z_n\inv }$, while any given element of $\bigoplus_{z\in Z} \Lambda_z$ has finite support in $Z$.

Condition~\eqref{pt:co-amen} is immediate if $Z$ is non-amenable, because if $\Lambda$ were co-amenable in $\Gamma$, then so would be the normal subgroup $\bigoplus_{z\in Z} \Lambda_z$, which is equivalent to the quotient $Z$ being amenable.

However, condition~\eqref{pt:co-amen} does indeed hold more generally as soon as $Z$ is non-trivial, e.g.\ $Z=\ZZ$. Suppose indeed for a contradiction that $\mu$ is a $\Gamma$-invariant mean on $\Gamma/\Lambda$. Consider a general element $\gamma\in \Gamma$ with coordinates $\gamma=((\lambda_z)_{z\in Z}, x)$. Given $y\in Z$, the stabiliser in $\Lambda_y$ of the point $\gamma \Lambda \in \Gamma/\Lambda$ is $\Lambda_y\cap \gamma \Lambda_e \gamma\inv$. On the other hand, $\gamma \Lambda_e \gamma\inv = \Lambda_x$ and therefore this stabiliser is trivial whenever $x\neq y$. In conclusion, all orbits of the $\Lambda_y$-action on $\Gamma/\Lambda$ are regular orbits except the orbits of the points $\gamma \Lambda$ where $\gamma=((\lambda_z)_{z\in Z}, y)$. Since $\mu$ is $\Lambda_y$-invariant and $\Lambda_y\cong \Lambda$ is non-amenable, $\mu$ is supported on the complement of the union of regular $\Lambda_y$-orbits. That is, $\mu$ is supported on the union of the orbits of the form $(*, y)\Lambda$. Applying the same argument to any $y'\neq y$, which exists since $Z$ is non-trivial, leads to a contradiction.

\medskip
We observe that when $Z=\ZZ$, this non-co-amenability contrasts with the co-amenability of $\bigoplus_{n\geq 0} \Lambda_z$ in $\Lambda \wr \ZZ$, see~\cite{Monod-Popa}.

\bigskip

For our final comment, we do consider topological groups to point out another source of exotic quasi-regular algebras but in the familiar context of Lie groups (which then of course will have non-abelian free subgroups). 

Let $G$ be a center-free non-compact connected simple Lie group of rank at least two. Let $\Lambda<G$ be any discrete subgroup which is not virtually soluble and not Zariski dense. Then we claim that $\cstar(\lambda_{G/\Lambda})$ is exotic.

The assumptions above are satisfied when $\Lambda<H<G$ is contained in proper connected Lie subgroup $H<G$ and is not virtually soluble. For instance, given any $2\leq n < m$ we can consider $\Lambda=\SL_n(\ZZ)$ and $G=\SL_m(\RR)$. Here $H$ is witnessed by $\SL_n(\RR)$.

In that example, $\cstar(\lambda_{\SL_m(\RR) / \SL_n(\ZZ)})$ is an exotic algebra for the \emph{Type~I} group $\SL_m(\RR)$. This cannot happen for discrete groups, since by Thoma's theorem~\cite{Thoma68} a discrete Type~I group is in particular amenable.

\smallskip

The proof of the exoticism claim follows again the same lines, but adapted to the setting of locally compact groups. The fact that $\Lambda$ is non-amenable follows from Tits's alternative~\cite{Tits72}. That it is not co-amenable follows from Zariski non-density by a result of Shalom~\cite[Cor.~1.7]{Shalom99}. Finally, the unconfinment was proved by Fraczyk--Gelander~\cite[Thm.~1.5]{Fraczyk-Gelander} assuming that $\Lambda$ is not a lattice, which in our case follows from Zariski non-density by Borel's density theorem \cite{Borel60}. These representations are again far from irreducible, in contrast to the Theorem.

\subsection*{Acknowledgements}
We are grateful to Mehrdad Kalantar and to Tim de Laat for their comments on the context of exotic algebras.

During this research M.G. was supported by the DFG -- Project-ID 427320536 -- SFB 1442, and under Germany's Excellence Strategy EXC 2044 390685587, Mathematics Münster: Dynamics--Geometry--Structure.


\bibliographystyle{amsalpha}
\bibliography{../BIB/ma_bib}

\def\cprime{$'$}
\providecommand{\bysame}{\leavevmode\hbox to3em{\hrulefill}\thinspace}
\providecommand{\MR}{\relax\ifhmode\unskip\space\fi MR }
\providecommand{\MRhref}[2]{%
  \href{http://www.ams.org/mathscinet-getitem?mr=#1}{#2}
}
\providecommand{\href}[2]{#2}
\begin{thebibliography}{LBMB18}

\bibitem[AD03]{Anantharaman03}
Claire Anantharaman-Delaroche, \emph{On spectral characterizations of
  amenability}, Israel J. Math. \textbf{137} (2003), 1--33.

\bibitem[BBRT21]{Brum-MatteBon-Rivas-Triestino_arx}
Joaqu{\'\i}n Brum, Nicol{\'a}s~Matte Bon, Crist{\'o}bal Rivas, and Michele
  Triestino, \emph{Locally moving groups acting on the line and
  $\mathbb{R}$-focal actions}, 2021, Preprint, arXiv 2104.14678v1.

\bibitem[Bek99]{Bekka99}
Mohammed El~Bachir Bekka, \emph{On the full {$C^*$}-algebras of arithmetic
  groups and the congruence subgroup problem}, Forum Math. \textbf{11} (1999),
  no.~6, 705--715.

\bibitem[BG13]{Brown-Guentner2013}
Nathanial~Patrick Brown and Erik~Paul Guentner, \emph{New {$\rm
  C^\ast$}-completions of discrete groups and related spaces}, Bull. Lond.
  Math. Soc. \textbf{45} (2013), no.~6, 1181--1193.

\bibitem[BK20]{Bekka-Kalantar}
Mohammed El~Bachir Bekka and Mehrdad Kalantar, \emph{Quasi-regular
  representations of discrete groups and associated {$C^*$}-algebras}, Trans.
  Amer. Math. Soc. \textbf{373} (2020), no.~3, 2105--2133.

\bibitem[BKLS96]{Bekka-Kaniuth-Lau-Schlichting}
Mohammed El~Bachir Bekka, Eberhard Kaniuth, Anthony To-Ming Lau, and G{\"u}nter
  Schlichting, \emph{On {$C^*$}-algebras associated with locally compact
  groups}, Proc. Amer. Math. Soc. \textbf{124} (1996), no.~10, 3151--3158.

\bibitem[BLR18]{Burillo-Lodha-Reeves}
Jos\'{e} Burillo, Yash Lodha, and Lawrence~David Reeves, \emph{Commutators in
  groups of piecewise projective homeomorphisms}, Adv. Math. \textbf{332}
  (2018), 34--56.

\bibitem[Bor60]{Borel60}
Armand Borel, \emph{Density properties for certain subgroups of semi-simple
  groups without compact components}, Ann. of Math. (2) \textbf{72} (1960),
  179--188.

\bibitem[Dix96]{Dixmier6996_C}
Jacques Dixmier, \emph{Les {$C^*$}-alg\`ebres et leurs repr\'esentations}, Les
  Grands Classiques Gauthier-Villars, \'Editions Jacques Gabay, Paris, 1996,
  Reprint of the second (1969) edition.

\bibitem[dLS21]{deLaat-Siebenand}
Tim de~Laat and Timo Siebenand, \emph{Exotic group {$C^*$}-algebras of simple
  {L}ie groups with real rank one}, Ann. Inst. Fourier (Grenoble) \textbf{71}
  (2021), no.~5, 2117--2136.

\bibitem[EQ99]{Echterhoff-Quigg}
Siegfried Echterhoff and John~Currie Quigg, \emph{Induced coactions of discrete
  groups on {$C^*$}-algebras}, Canad. J. Math. \textbf{51} (1999), no.~4,
  745--770. \MR{1701340}

\bibitem[Eym72]{Eymard72}
Pierre Eymard, \emph{Moyennes invariantes et repr\'esentations unitaires},
  Springer-Verlag, Berlin, 1972, Lecture Notes in Mathematics, Vol. 300.

\bibitem[Fel64]{Fell1964}
James Michael~Gardner Fell, \emph{Weak containment and induced representations
  of groups. {II}}, Trans. Amer. Math. Soc. \textbf{110} (1964), 424--447.

\bibitem[FG23]{Fraczyk-Gelander}
Mikolaj Fraczyk and Tsachik Gelander, \emph{Infinite volume and infinite
  injectivity radius}, Ann. of Math. (2) \textbf{197} (2023), no.~1, 389--421.

\bibitem[God48]{Godement48}
Roger Godement, \emph{Les fonctions de type positif et la th\'{e}orie des
  groupes}, Trans. Amer. Math. Soc. \textbf{63} (1948), 1--84.

\bibitem[HZ93]{Hartley-Zalesskii93}
Brian Hartley and Aleksandr~Efimovich Zalesski\u{\i}, \emph{On simple periodic
  linear groups---dense subgroups, permutation representations, and induced
  modules}, Israel J. Math. \textbf{82} (1993), no.~1-3, 299--327.

\bibitem[KLQ13]{Kaliszewski-Landstad-Quigg13}
Steven~Paul Kaliszewski, Magnus~Brostrup Landstad, and John~Currie Quigg,
  \emph{Exotic group {$C^*$}-algebras in noncommutative duality}, New York J.
  Math. \textbf{19} (2013), 689--711. \MR{3141810}

\bibitem[KLQ16]{Kaliszewski-Landstad-Quigg16}
\bysame, \emph{Exotic coactions}, Proc. Edinb. Math. Soc. (2) \textbf{59}
  (2016), no.~2, 411--434.

\bibitem[KS12]{Kyed-Soltan}
David Kyed and Piotr~Miko{\l}aj So{\l}tan, \emph{Property ({T}) and exotic
  quantum group norms}, J. Noncommut. Geom. \textbf{6} (2012), no.~4, 773--800.

\bibitem[LBMB18]{LeBoudec-MatteBon_C*s}
Adrien Le~Boudec and Nicol{\'a}s Matte~Bon, \emph{Subgroup dynamics and
  {{\(C^\ast\)}}-simplicity of groups of homeomorphisms}, Ann. Sci. {\'E}c.
  Norm. Sup{\'e}r. (4) \textbf{51} (2018), no.~3, 557--602.

\bibitem[Mac51]{Mackey51}
George~Whitelaw Mackey, \emph{On induced representations of groups}, Amer. J.
  Math. \textbf{73} (1951), 576--592.

\bibitem[Mon13]{Monod_PNAS}
Nicolas Monod, \emph{Groups of piecewise projective homeomorphisms}, Proc.
  Natl. Acad. Sci. USA \textbf{110} (2013), no.~12, 4524--4527.

\bibitem[Mon23]{Monod_piecewise_pre}
\bysame, \emph{Some comments on piecewise-projective groups of the line},
  Preprint, {\tt https://doi.org/10.48550/arXiv.2305.00796}, 2023.

\bibitem[MP03]{Monod-Popa}
Nicolas Monod and Sorin Popa, \emph{On co-amenability for groups and von
  {N}eumann algebras}, C. R. Math. Acad. Sci. Soc. R. Can. \textbf{25} (2003),
  no.~3, 82--87.

\bibitem[Oka14]{Okayasu14}
Rui Okayasu, \emph{Free group {$C^*$}-algebras associated with {$\ell_p$}},
  Internat. J. Math. \textbf{25} (2014), no.~7, 1450065, 12.

\bibitem[RW16]{Ruan-Wiersma}
Zhong-Jin Ruan and Matthew Wiersma, \emph{On exotic group {$\rm
  C^*$}-algebras}, J. Funct. Anal. \textbf{271} (2016), no.~2, 437--453.

\bibitem[Sha99]{Shalom99}
Yehuda Shalom, \emph{Invariant measures for algebraic actions, {Z}ariski dense
  subgroups and {K}azhdan's property ({T})}, Trans. Amer. Math. Soc.
  \textbf{351} (1999), no.~8, 3387--3412.

\bibitem[Tho68]{Thoma68}
Elmar Thoma, \emph{Eine {C}harakterisierung diskreter {G}ruppen vom {T}yp {I}},
  Invent. Math. \textbf{6} (1968), 190--196.

\bibitem[Tit72]{Tits72}
Jacques Tits, \emph{Free subgroups in linear groups}, J. Algebra \textbf{20}
  (1972), 250--270.

\bibitem[Val84]{Valette84}
Alain Valette, \emph{Minimal projections, integrable representations and
  property {$({\rm T})$}}, Arch. Math. (Basel) \textbf{43} (1984), no.~5,
  397--406.

\bibitem[Wie16]{Wiersma2016}
Matthew Wiersma, \emph{Constructions of exotic group {$C^*$}-algebras},
  Illinois J. Math. \textbf{60} (2016), no.~3-4, 655--667.

\end{thebibliography}

\end{document}